\documentclass[12pt]{article}
\usepackage{amsmath,amsfonts,amsthm,amssymb}
\usepackage{graphicx}

\newtheorem{theorem}{Theorem}[section]

\newtheorem{corollary}{Corollary}[section]

\newtheorem{lemma}[theorem]{Lemma}

\newtheorem{rque}{Remark}

\newtheorem{definition}{Definition}[section]
\newtheorem{exple}[theorem]{Example}

\title{A Novel Extension of $S$-Metric Spaces with Application to Solving $n$th-Degree Polynomial Equations}

\author {Nizar Souayah$^{1}$\\
  $^{1}$Department of Natural Sciences, Community College Al-Riyadh,\\  King Saud University, KSA\\
  \\E-mail:    nsouayah@ksu.edu.sa$^{1}$\\
    							}

\date{}
\begin{document}

\maketitle{}

\begin{abstract}
This paper introduces a novel generalization of the classical concept of $S$-metric spaces, referred to as composed $S$-metric spaces. By incorporating a composed function, we impose more general conditions on the triangle inequality, extending the traditional framework. Using this newly defined structure, we establish the existence and uniqueness of fixed points under suitable assumptions. Our results expand upon and generalize several established results in the literature across various types of spaces. To highlight the practical significance  of our findings, we provide an example that demonstrate the applicability of the theoretical results. This examples highlight the versatility and effectiveness of composed $S$-metric spaces in a range of mathematical contexts, particularly in solving $n$th degree polynomial equations.
\end{abstract}


\section{Introduction and Preliminaries}
Fixed point (FP for short) theory, a cornerstone of modern mathematical analysis, plays a crucial role in a wide range of disciplines, including topology, nonlinear analysis, and optimization \cite{a1,a2,a3,a4,a5,a6,a7,a8,a9}. Results concerning fixed points have profound implications in the study of dynamical systems, differential equations, and even economics. Over the years, FP results in metric spaces have led to fundamental advances, with the classical Banach FP Theorem \cite{banach} being one of the most significant contributions to the field.\\
While classical FP results, such as those in metric spaces and more specialized spaces like Banach and Hilbert spaces, provide a solid foundation, they do not always cover more complex or generalized settings where the distance between points might be more nuanced. This gap has led to the development of generalized metric spaces, such as $S$-metric spaces (for short $S$-$MS$). As an extension of traditional metric spaces, $S$-$MS$ have gained significant attention in recent years due to their ability to provide a more generalized framework for analyzing metric functions. An early investigation into $S$-$MS$ can be found in the work of Sedghi \cite{sedghi}, who first outlined the foundational properties of these spaces and explored their basic structure. 
\begin{definition}\cite{sedghi}
Let $\mathcal{A}$ be a nonempty set. An $S$-metric on $\mathcal{A}$ is a function \\$S_m : \mathcal{A}^3 \longrightarrow [0,\infty)$ that fulfills  the following hypothesis, for all $q,h,z,u \in \mathcal{A}$ :
\begin{itemize}
	\item[(i)]  $S_m(q,h,z) = 0 \mbox{  if and only if   }   q = h = z$.
	\item[(ii)] $S_m(q,h,z) \leq S_m(q,q,u) +S_m(h,h,u) +S_m(z,z,u)$.
\end{itemize}
The pair $(\mathcal{A},S_m)$ is called an $S$-metric space.
\end{definition}
\begin{exple}
Let $\mathcal{A}=\mathbb{R}$ and $||.||$ be a norm on $\mathcal{A}$. Then
\begin{equation*}
S_m(q,h,z) =||q+h-2z||+ ||h-z||
\end{equation*}
is an $S$-metric space on $\mathcal{A}$. 
\end{exple}
Since then, numerous researchers have contributed to the development of this field by producing  several types of generalizations and examining the properties and  studying the existence and uniqueness of FP  \cite{r1,r2,r3,r4,r5,r9,r10,r11,r12}. \\
To set the context, we begin by reviewing some fundamental definitions of some generalizations of $S$-$MS$ pertinent for our work.\\ 
In \cite{r6}, the authors introduced the concept of $S_b$-metric spaces which is a generalization of the $b$-metric spaces.
\begin{definition}\cite{r6}
Let $\mathcal{A}\neq \emptyset$ and let $b\geq 1$ be a given real number. A function  $S_{b}:\mathcal{A}^{3}\rightarrow [0,\infty)$ is said to be $S_b$-metric  if and only if  for all $q,h,z,u \in \mathcal{A}:$ the following conditions hold:
\begin{itemize}
\item[(i)]  $S_{b}(q,h,z)= 0$ if and only if $q=h=z$.
\item[(ii)] $S_{b}(q,h,z)\leq   b[S_{b}(q,q,u)+S_{b}(h,h,u)+S_{b}(z,z,u)]$.
\end{itemize}
The pair $(\mathcal{A},S_{b})$ is called an $S_b$-metric space.
\end{definition}

\begin{exple}
Let $\mathcal{A}\neq \emptyset$  and $card(\mathcal{A})\geq 5$. Suppose $\mathcal{A}=\mathcal{A}_1\cup \mathcal{A}_2$ a partition of $\mathcal{A}$ such that $card(\mathcal{A}_1)\geq 4$. Let $b\geq 1$. Then
\begin{equation*}
 S_b(a,b,c)= \left\{
      \begin{aligned}
       & 0  & \mbox{ if }  & q=h=z=0\\
       & 3b & \mbox{ if }  & (q,h,z)\in \mathcal{A}_1^3\\
       & 1  & \mbox{ if }  & (q,h,z)\notin \mathcal{A}_1^3 \\
      \end{aligned}
    \right.
\end{equation*}
$S_b$ is an $S_b$-metric on $\mathcal{A}$ with coefficient $b\geq 1$ for all $q,h,z\in \mathcal{A}$.
\end{exple}

Motivated by the concept of the  $S_b$-metric spaces, Sedghi et al. \cite{Sp} established the notion of  $S_p$-metric spaces as follows:
 \begin{definition}\cite{Sp}
Let $\mathcal{A}\neq \emptyset$ and  $\Gamma: [0,\infty)\rightarrow [0,\infty)$ be a strictly increasing function and continuous, where $a\leq \Gamma(a)$ for all $a>0$ and $\Gamma(0)=0$.  A function  $S_{p}:\mathcal{A}^{3}\rightarrow [0,\infty)$ is said to be $S_p$-metric  if and only if  for all $q,h,z,u \in \mathcal{A}:$ the following conditions hold:
\begin{itemize}
\item[(i)]  $S_{p}(q,h,z)= 0$ if and only if $q=h=z$.
\item[(iii)] $S_{p}(q,h,z)\leq   \Gamma \Big(S_{p}(q,q,u)+S_{p}(h,h,u)+S_{p}(z,z,u)\Big)$.
\end{itemize}
The pair $(\mathcal{A},S_{p})$ is called a $S_p$-metric space.
\end{definition}

\begin{rque}
Each $S$-$MS$ is an $S_p$-metric space with $\Gamma(a) = a$ and every $S_b$-metric space with parameter $b \geq 1$ is an $S_p$-metric space with $\Gamma(a) = b.a$.
\end{rque}

\begin{exple}
Let $(\mathcal{A},S)$ be an $S$-$MS$ with coefficient $b\geq 1$. Then 
\begin{enumerate}
	\item $S_{p}(q,h,z)=\exp(S(q,h,z))-1$ is an  $S_p$-metric space with $\Gamma(a)=\exp(a)-1$.
	\item $S_{p}(q,h,z)=\exp(S(q,h,z)) \ln(1+S(q,h,z))$ is an  $S_p$-metric space with $\Gamma(a)=\exp(b.a)\ln(1+b.a)$.
\end{enumerate}
\end{exple}

Recently, the concept of  $S$-$MS$ was  expanded to include  controlled $S$-metric-like spaces \cite{r7} by applying a control function   into the triangle inequality, as showed in the following definition:
\begin{definition}\cite{r7}
Consider $\mathcal{A}\neq \emptyset$ and a function $\alpha : \mathcal{A}^2 \rightarrow [1,\infty)$. If a function
$S : \mathcal{A} \rightarrow [0,\infty)$  meets the following conditions for every   $q, h, z,a\in  \mathcal{A}$, then  the pair $(\mathcal{A},S)$ is called a controlled $S$-metric type space.
\begin{itemize}
\item (ii) $S(q, h, z) = 0 \Longleftrightarrow  q = h = z$;
\item(iii) $S(q, h, z)\leq \alpha(q, a)S(q, q, a) + \alpha(h, a)S(h, h, a) + \alpha(z, a)S(z, z, a)$. 
\end{itemize}
\end{definition}

Inspired by this result, we aim to extend the concept of a controlled $S$-metric space to a new metric space called composed $S$-metric spaces  by incorporating composed functions into the right-hand side of the triangle inequality. As a result, some of the findings discussed in the literature will emerge as special cases of the results derived in the composed $S$-metric spaces. Moreover, This generalization allows for a more flexible structure, accommodating a broader class of spaces and functions. The precise definition and formulation of this new concept are provided in the following section. 



\section{Composed $S$-metric spaces}
\begin{definition}\label{def_CS}
Let $\mathcal{A}$ be a non empty set and consider the  non-constant functions $\alpha: [0, \infty) \rightarrow[0, \infty)$. A function $C_S: \mathcal{A}^3 \rightarrow[0, \infty)$ is said to be a  Composed $S$-metric if it satisfies for all $q,h,w,u\in \mathcal{A}$:
\begin{enumerate}
\item\label{con:1} $C_S(q,h,w)=0$ $\Longleftrightarrow$ $q=h=w=0$ $\forall$ $q,h,w \in \mathcal{A}$,
\item $C_S(q,h,w) \leq \alpha\big( C_S(q,q,u)\big)+\alpha\big(C_S(h,h,u)\big)+\alpha\big(C_S(w,w,u)\big)$ .
\end{enumerate}
 The pair $\left(\mathcal{A}, C_S\right)$ is a called  Composed $S$-metric space ($CSM$ for short).
\end{definition}

\begin{definition}\label{def_CS}
Let  $\left(\mathcal{A}, C_S\right)$ be a $CSM$.  $\mathcal{A}$ is said to be symmetric if $C_S(q,q,h)=C_S(h,h,q)$  for all  $q,h\in \mathcal{A}$. 
\end{definition}

\begin{rque}
We observe that every $CSM$ is an $S$-$MS$ when $\alpha(q)=q$, but the reverse is not necessarily true, as demonstrated in Example  \ref{exple1}.
\end{rque}

\begin{exple}\label{exple1}
Let $\mathcal{A}=[1,+\infty)$. Define a function $C_S:\mathcal{A}^3\longrightarrow [0,\infty)$ by
\begin{equation}
C_S(q,h,w)=(q-w)^2+(h-w)^2. \label{smetric}
\end{equation}
Define  the composed function $\alpha:[1,+\infty)\longrightarrow [0,\infty)$ by $\alpha(q)=e^q$.\\
Then, $\left(\mathcal{A}, C_S\right)$ is a $CSM$ but is not an $S$-metric space in the usual sense.
\end{exple}
\begin{proof}
We begin by proving that the function $C_S$ is not an $S$-metric in the usual sense since the triangle inequality is not satisfied.  Indeed, for $q=4$, $h=5$,  $w=1$ and $t=4$ we have $C_S(q,h,w)=(q-w)^2+(h-w)^2=25 >C_S(q,q,t)+ C_S(h,h,t)+C_S(w,w,t)=20$.\\
Verifying the first condition $1.$ of Definition \ref{def_CS} is straightforward. We should prove condition $2$. For all $q,h,w,u\in [1,\infty)$ we have
\begin{equation*}
C_S(q,h,w)= (q-w)^2+(h-w)^2=(q-u+u-w)^2+(h-u+u-w)^2.
\end{equation*}
By using  the classical inequality $(a+b)^2\leq 2(a^2+b^2)$ we obtain
\begin{eqnarray}
C_S(q,h,w) & = & (q-u+u-w)^2+(h-u+u-w)^2 \nonumber \\
           & \leq & 2(q-u)^2+2(u-w)^2+2(h-u)^2+2(w-u)^2\nonumber\\
           & \leq & 2(q-u)^2+2(h-u)^2+4(w-u)^2.\label{ex1}
\end{eqnarray}
Knowing that for all $q\geq 1$, $q\leq e^q$ and $4q\leq e^{2q}$, we get from (\ref{ex1})
\begin{eqnarray}
C_S(q,h,w)   & \leq & e^{2(q-u)^2}+e^{2(h-u)^2}+e^{2(w-u)^2}\nonumber \\
             & = & \alpha\big(2(q-u)^2 \big)+\alpha\big(2(h-u)^2 \big)+\alpha\big(2(w-u)^2 \big)\nonumber \\
						 & = & \alpha\big(C_S(q,q,u) \big)+\alpha\big(C_S(h,h,u) \big)+\alpha\big(C_S(w,w,u) \big).\nonumber
\end{eqnarray}
This validates the triangle inequality.
\end{proof}

\begin{exple}
Suppose $\mathcal{A}=\mathbb{N}$ and let  $C_S:\mathcal{A}^3\longrightarrow [0,\infty)$ defined by
$$
 \left\{
    \begin{array}{ll}
        C_S(c,c,c)=0 & \forall c\in \mathbb{N}, \\
				C_S(c,c,e)=c+e & \forall c,e\in \mathbb{N},\\
        C_S(c,d,e)=2(c+d+e) &   \mbox{ for all distinct } c,d,e\in \mathbb{N}.
    \end{array}
\right.
$$
Define  the composed function $\alpha:[0,+\infty)\longrightarrow [0,\infty)$ by $\alpha(t)=2t+1$.\\
Then, $\left(\mathcal{A}, C_S\right)$ is a $CSM$ but is not an $S$-$MS$.
\end{exple}
\begin{proof}
It is clear that the condition referred to the self distance is valid. Let $c,d,k,t\in \mathbb{N}$, we have 
\begin{eqnarray*}
 \alpha\big(C_S(c,c,t)\big)+\alpha\big(C_S(d,d,t)\big)+\alpha\big(C_S(e,e,t)\big) & = & \alpha(c+t)+\alpha(d+t)+\alpha(e+t) \\
                                      & = & 2(c+t)+1 +2(d+t)+1+2(e+t)+1\\
																			& \geq & 2(c+d+e)=C_S(c,d,e).
\end{eqnarray*} 
Therefore, $(\mathcal{A}, C_S)$ is a $CSM$. However, it is not an $S$-$MS$ in the usual sense. Indeed, the triangle inequality is not satisfied for all $c,d,e,t\in \mathbb{N}$ since 
\begin{eqnarray*}
C_S(c,c,t)+C_S(d,d,t)+C_S(e,e,t) & = & c+d+e+2t \\
                                 & \not\geq 2 & ( c+d+e) =C_S(c,d,e) \mbox{ for } c=d=e=t=1.
\end{eqnarray*}
\end{proof}

\begin{exple}\label{exple3}
Let $\mathcal{A}=[1,+\infty)$. Define a function $C_S:\mathcal{A}^3\longrightarrow [0,\infty)$ by
\begin{equation}
C_S(q,h,w)=|q-w|+|h-w|. \label{eq-ex3}
\end{equation}
Define  the composed function $\alpha:[1,+\infty)\longrightarrow [0,\infty)$ by $\alpha(q)=e^{2q}$.\\
Then, $\left(\mathcal{A}, C_S\right)$ is a $CSM$.
\end{exple}

\begin{definition}
Let $(\mathcal{A},C_S)$ be a $CSM$ and $\{q_{n}\}$ be a sequence in $\mathcal{A}$. Then
\begin{itemize}
\item[(i)]  A sequence $\{q_{n}\}$ is called convergent if and only if there exists $q\in \mathcal{A}$ such that $C_S(q_n,q_n,q)\longrightarrow 0$  as $n\longrightarrow \infty$. That is for each $\epsilon >0$ there exists $n_0\in \mathbb{N}$ such that $\forall n\geq n_0$, $C_S(q_n, q_n, q) < \epsilon$ and we write $\displaystyle \lim_{n\longrightarrow \infty}q_n=q$.
\item[(ii)] A sequence $\{q_{n}\}$ is called a Cauchy sequence if and only if $C_S(q_n,q_n,q_m)\longrightarrow 0$ as $n,m\longrightarrow \infty$. That is for each $\epsilon >0$ there exists $n_0\in \mathbb{N}$ such that $\forall n,m\geq n_0$,
$C_S(q_n, q_n, q_m) < \epsilon$. 
\item[(iii)] $(\mathcal{A},C_S)$ is considered a complete $CSM$ if every Cauchy sequence within it converges.
\end{itemize}
\end{definition}

\begin{lemma}
Let $(\mathcal{A},C_S)$ be a $CSM$ where the composed function $\alpha$ satisfies $\alpha(0)=0$ and $\{q_{n}\}$ be a convergent sequence in $\mathcal{A}$, then the limit is unique.
\end{lemma}
\begin{proof}
Let $\{q_{n}\}$ be a convergent sequence in $\mathcal{A}$ and suppose that converges to both $q_1$ and $q_2$ (with $q_1\neq q_2$) that is $\displaystyle \lim_{n\longrightarrow \infty}C_S(q_n, q_n, q_1)=0$ and  $\displaystyle \lim_{n\longrightarrow \infty}C_S(q_n, q_n, q_2)=0$.   we have
\begin{equation}
C_S(q_1, q_1, q_2)  \leq   2\alpha\big(C_S(q_1, q_1, q_n) \big) +\alpha\big(C_S(q_2, q_2, q_n) \big). \label{lemma1}
\end{equation}
Taking the limit as $n$ tends to $\infty$ in equation (\ref{lemma1}) we obtain that 
\begin{equation*}
C_S(q_1, q_1, q_2)  \leq  2\alpha(0 ) +\alpha(0)=0. 
\end{equation*}
Then, $C_S(q_1, q_1, q_2)=0$ and $q_1=q_2$.
\end{proof}

In the following section we  explore several FP theorems and an application within this extended framework.

\section{Fixed Point results}
The first main result presented in this paper is the establishment of a generalized FP theorem for composed $S$-metric spaces, which builds upon the Banach contraction principle.
\begin{theorem}\label{thm1}
Let $(\mathcal{A},C_S)$ be a complete $CSM$ with non-constant control function  $\alpha:[0,\infty)\longrightarrow [0,\infty)$. Let $F: \mathcal{A} \rightarrow \mathcal{A}$ be a  mapping  such that: 
\begin{equation}
C_S(Fq,Fh,Fw)\leq r C_S(q,h,w) \  \ \forall q,h,w\in \mathcal{A} \mbox{ and } r\in (0,1). \label{banach-contr}
\end{equation}
For $\bold{I}_0\in \mathcal{A}$, define a sequence $\{\bold{I}_n\}$ by $\bold{I}_n=F^n\bold{I}_0$. Suppose that the following conditions are satisfied:
\begin{enumerate}
	\item  $\alpha (ks+t) \leq k\alpha(s)+\alpha(t)$ for all positive number  $s, t$ and $k$.
	\item  $\alpha(0)=0$.
	\item $\displaystyle\lim_{n\rightarrow \infty}\displaystyle\sum_{k=n+3}^{m-2}2^{k-n-1}\alpha^{k-n+1}\Big(r^kC_S(\bold{I}_0,\bold{I}_0,\bold{I}_{1})\Big)+2^{m-n-2}\alpha^{m-n-1}\Big(r^mC_S(\bold{I}_0,\bold{I}_0,\bold{I}_{1})\Big)=0$.
\end{enumerate}
Then, $F$ has a unique FP.
\end{theorem}

\begin{proof}
Consider the sequence $\{\bold{I}_n\}\in \mathcal{A}$ defined in the statement of the theorem. We have 
\begin{equation}
C_S(\bold{I}_n,\bold{I}_n,\bold{I}_{n+1})=C_S(F^n\bold{I}_0,F^n\bold{I}_0,F^n\bold{I}_1) \leq r^n C_S(\bold{I}_0,\bold{I}_0,\bold{I}_1)\forall n\in \mathbb{N}.\label{eq1}
\end{equation}
Let $n,m\in\mathbb{N}$ such that $n<m$. Using the triangle inequality we get
\begin{eqnarray*}
C_S(\bold{I}_n,\bold{I}_n,\bold{I}_{m}) & \leq &  \alpha(C_S(\bold{I}_n,\bold{I}_n,\bold{I}_{n+1}))+\alpha(C_S(\bold{I}_n,\bold{I}_n,\bold{I}_{n+1}))+\alpha(C_S(\bold{I}_m,\bold{I}_m,\bold{I}_{n+1}))\\
                    & = &     2\alpha\big(C_S(\bold{I}_n,\bold{I}_n,\bold{I}_{n+1})\big)+\alpha\big(C_S(\bold{I}_m,\bold{I}_m,\bold{I}_{n+1})\big)\\
									  & \leq &  2\alpha(C_S(\bold{I}_n,\bold{I}_n,\bold{I}_{n+1}))+ \alpha\Big(2\alpha(C_S(\bold{I}_m,\bold{I}_m,\bold{I}_{n+2})+\alpha(C_S(\bold{I}_{n+1},\bold{I}_{n+1},\bold{I}_{n+2}))\Big)\\
										& \leq &  2\alpha(C_S(\bold{I}_n,\bold{I}_n,\bold{I}_{n+1}))+ 2\alpha^2\Big(C_S(\bold{I}_m,\bold{I}_m,\bold{I}_{n+2})\Big)+\alpha^2\Big(C_S(\bold{I}_{n+1},\bold{I}_{n+1},\bold{I}_{n+2})\Big)\\
										& \leq &  2\alpha(C_S(\bold{I}_n,\bold{I}_n,\bold{I}_{n+1}))+\alpha^2\Big(C_S(\bold{I}_{n+1},\bold{I}_{n+1},\bold{I}_{n+2})\Big)\\
										& + &   2\alpha^2\Big(2\alpha(C_S(\bold{I}_m,\bold{I}_m,\bold{I}_{n+3})) +\alpha(C_S(\bold{I}_{n+2},\bold{I}_{n+2},\bold{I}_{n+3}))  \Big)\\
										& \leq & 2\alpha(C_S(\bold{I}_n,\bold{I}_n,\bold{I}_{n+1}))+\alpha^2\Big(C_S(\bold{I}_{n+1},\bold{I}_{n+1},\bold{I}_{n+2})\Big)\\
										& + &   2\alpha^3\big(C_S(\bold{I}_{n+2},\bold{I}_{n+2},\bold{I}_{n+3})\big)+2^2\alpha^3\Big(C_S(\bold{I}_m,\bold{I}_m,\bold{I}_{n+3})  \Big)\\
										& \leq & 2\alpha(C_S(\bold{I}_n,\bold{I}_n,\bold{I}_{n+1}))+\alpha^2\Big(C_S(\bold{I}_{n+1},\bold{I}_{n+1},\bold{I}_{n+2})\Big)\\
										& + &   2\alpha^3\big(C_S(\bold{I}_{n+2},\bold{I}_{n+2},\bold{I}_{n+3})\big)+2^2\alpha^3\Big[2\alpha\big(C_S(\bold{I}_m,\bold{I}_m,\bold{I}_{n+4}) \big) \\
										& + &    \alpha(C_S(\bold{I}_{n+3},\bold{I}_{n+3},\bold{I}_{n+4}))\Big]\\
										& \leq & 2\alpha(C_S(\bold{I}_n,\bold{I}_n,\bold{I}_{n+1}))+\alpha^2\Big(C_S(\bold{I}_{n+1},\bold{I}_{n+1},\bold{I}_{n+2})\Big) + 2\alpha^3\big(C_S(\bold{I}_{n+2},\bold{I}_{n+2},\bold{I}_{n+3})\big)\\
										& + & 2^2\alpha^4(C_S(\bold{I}_{n+3},\bold{I}_{n+3},\bold{I}_{n+4}))+2^3\alpha^4\big(C_S(\bold{I}_m,\bold{I}_m,\bold{I}_{n+4}) \big)\\
										& \leq &  2\alpha(C_S(\bold{I}_n,\bold{I}_n,\bold{I}_{n+1}))+\alpha^2\Big(C_S(\bold{I}_{n+1},\bold{I}_{n+1},\bold{I}_{n+2})\Big) + 2\alpha^3\big(C_S(\bold{I}_{n+2},\bold{I}_{n+2},\bold{I}_{n+3})\big)\\
										& + & 2^2\alpha^4(C_S(\bold{I}_{n+3},\bold{I}_{n+3},\bold{I}_{n+4}))+2^3\alpha^4\Big[2\alpha\big(C_S(\bold{I}_m,\bold{I}_m,\bold{I}_{n+5}) \big) \\
										& + & \alpha\big(C_S(\bold{I}_{n+4},\bold{I}_{n+4},\bold{I}_{n+5})\big) \Big]\\
										& \vdots &\\
										& \leq & 2\alpha(C_S(\bold{I}_n,\bold{I}_n,\bold{I}_{n+1}))+\alpha^2\Big(C_S(\bold{I}_{n+1},\bold{I}_{n+1},\bold{I}_{n+2})\Big) + 2\alpha^3\big(C_S(\bold{I}_{n+2},\bold{I}_{n+2},\bold{I}_{n+3})\big)\\ 
										& + & \displaystyle\sum_{k=n+3}^{m-2}2^{k-n-2}\alpha^{k-n+1}\Big(C_S(\bold{I}_k,\bold{I}_k,\bold{I}_{k+1})\Big)+2^{m-n-3}\alpha^{m-n-1}\Big(C_S(\bold{I}_m,\bold{I}_m,\bold{I}_{m-1})\Big).										
\end{eqnarray*}
Using (\ref{eq1}), we obtain
\begin{eqnarray*}
C_S(\bold{I}_n,\bold{I}_n,\bold{I}_{m}) & \leq & 2\alpha(r^nC_S(\bold{I}_0,\bold{I}_0,\bold{I}_{1}))+\alpha^2\Big(r^{n+1}C_S(\bold{I}_0,\bold{I}_0,\bold{I}_{1})\Big) + 2\alpha^3\big(r^{n+2}C_S(\bold{I}_0,\bold{I}_0,\bold{I}_{1})\big)\\ 
										& + & \displaystyle\sum_{k=n+3}^{m-2}2^{k-n-1}\alpha^{k-n+1}\Big(r^kC_S(\bold{I}_0,\bold{I}_0,\bold{I}_{1})\Big)+2^{m-n-2}\alpha^{m-n-1}\Big(r^{m-1}C_S(\bold{I}_0,\bold{I}_0,\bold{I}_{1})\Big).	
\end{eqnarray*}
By condition (3) of the theorem and knowing that $0<r<1$  and $\alpha(0)=0$ we can affirm that $\displaystyle\lim_{n,m\rightarrow \infty}C_S(\bold{I}_n,\bold{I}_n,\bold{I}_{m})=0$. Therefore, $\{ \bold{I}_n\}$ is a Cauchy sequence. As $(\mathcal{A}, C_S)$ is  a complete $CMS$, it follows that $\bold{I}_n\rightarrow \bold{I}\in \mathcal{A}$ that is 
\begin{equation}
\displaystyle\lim_{n\rightarrow \infty}C_S(\bold{I}_n,\bold{I}_n,\bold{I})=0 \label{eq2}.
\end{equation}
We will establish that  $\bold{I}$ is a fixed point of $F$.  By the triangle inequality we have
\begin{eqnarray}
C_S(\bold{I},\bold{I},F\bold{I}) & \leq & 2\alpha(C_S(\bold{I},\bold{I},\bold{I}_{n+1}))+\alpha(C_S(F\bold{I},F\bold{I},\bold{I}_{n+1}))\nonumber \\
                           & = &    2\alpha(C_S(\bold{I},\bold{I},\bold{I}_{n+1}))+\alpha(C_S(F\bold{I},F\bold{I},F\bold{I}_{n}))\nonumber \\
													 & \leq & 2\alpha(C_S(\bold{I},\bold{I},\bold{I}_{n+1}))+\alpha(rC_S(\bold{I},\bold{I},\bold{I}_{n})).\label{eq3}
\end{eqnarray}
Taking the limit of inequality (\ref{eq3}) when $n\longrightarrow \infty$, gives us
\begin{equation*}
C_S(\bold{I},\bold{I},F\bold{I}) \leq 2\alpha(0)+ \alpha(0).
\end{equation*}
Knowing that  $\alpha(0)=0$ we can conclude that  $C_S(\bold{I},\bold{I},F\bold{I})=0$ that is $F\bold{I}=\bold{I}$, therefore $\bold{I}$ is a FP of $F$.\\
Let $\bold{I}_1,\bold{I}_2$ two FP of $F$. we have 
\begin{equation*}
C_S(\bold{I}_1,\bold{I}_1,\bold{I}_2)=C_S(F\bold{I}_1,F\bold{I}_1,F\bold{I}_2) \leq r C_S(\bold{I}_1,\bold{I}_1,\bold{I}_2).
\end{equation*}
Since $r<1$, $C_S(\bold{I}_1,\bold{I}_1,\bold{I}_2)=0$, therefore $\bold{I}_1=\bold{I}_2$.
\end{proof}
 Now, we introduce a family of function to investigate some FP theorems on $CSM$. Let $\mathcal{M}_f$ be the set of all continuous functions $M_f:\mathbb{R}_+^5\longrightarrow \mathbb{R}_+$ meeting the following conditions for some $r\in [0,1)$:
\begin{itemize}
	\item[$(M_1)$] $\forall$ $o,h,w\in \mathbb{R}_+$, if $h\leq M_f(o,o,0,w,h)$ with $w\leq 2o+h$, then $h\leq r.o$,
	\item[$(M_2)$] $\forall$ $h\in \mathbb{R}_+$, if $h\leq M_f(h,0,h,h,0)$ then $h=0$.
\end{itemize}

\begin{theorem}\label{thm2}
Let $(\mathcal{A},C_S)$ be a  symmetric complete $CSM$ with non-constant control function  $\alpha:[0,\infty)\longrightarrow [0,\infty)$. Let $F: \mathcal{A} \rightarrow \mathcal{A}$ be a  mapping  such that for all $o,h\in \mathcal{A}$ and $M_f\in \mathcal{M}_f$:
\begin{equation}
C_S(Fo,Fo,Fh)\leq M_f\big(C_S(o,o,h),C_S(Fo,Fo,o), C_S(Fo,Fo,h), C_S(Fh,Fh,o), C_S(Fh,Fh,h)  \big). \label{M-contr}
\end{equation}
For $\bold{I}_0\in \mathcal{A}$, take $\bold{I}_n=F^n\bold{I}_0$. Assume that the following conditions hold:
\begin{enumerate}
	\item[i)]  $\alpha (ks+t) \leq k\alpha(s)+\alpha(t)$ for all for all positive number  $s, t$ and $k$.
	\item[ii)]  $\alpha(0)=0$.
	\item[iii)] $\alpha(C_S(\bold{I}_n,\bold{I}_n,\bold{I}_{n+1}))\leq C_S(\bold{I}_n,\bold{I}_n,\bold{I}_{n+1}) $ for all $n\in \mathbb{N}$.
	\item[iv)] $\displaystyle\lim_{n\rightarrow \infty}\displaystyle\sum_{k=n+3}^{m-2}2^{k-n-1}\alpha^{k-n+1}\Big(r^kC_S(\bold{I}_0,\bold{I}_0,\bold{I}_{1})\Big)+2^{m-n-3}\alpha^{m-n-1}\Big(r^mC_S(\bold{I}_0,\bold{I}_0,\bold{I}_{1})\Big)=0$.
\end{enumerate}
Hence, $F$ has a unique FP.
\end{theorem}

\begin{proof}
Let $\bold{I}_{n}$ the sequence defined in the theorem. Using (\ref{M-contr}) and the symmetry of the metric we get:
\begin{eqnarray}
C_S(\bold{I}_{n+1},\bold{I}_{n+1},\bold{I}_{n+2}) & = & C_S(F\bold{I}_{n},F\bold{I}_{n},F\bold{I}_{n+1})\nonumber \\
                             & \leq & M_f\Big( C_S(\bold{I}_{n},\bold{I}_{n},\bold{I}_{n+1}), C_S(\bold{I}_{n+1},\bold{I}_{n+1},\bold{I}_{n}) ,C_S(\bold{I}_{n+1},\bold{I}_{n+1},\bold{I}_{n+1}),\nonumber \\
														 &  & C_S(\bold{I}_{n+2},\bold{I}_{n+2},\bold{I}_{n}) , C_S(\bold{I}_{n+2},\bold{I}_{n+2},\bold{I}_{n+1})\Big)\nonumber \\
														 & = & M_f\Big( C_S(\bold{I}_{n},\bold{I}_{n},\bold{I}_{n+1}), C_S(\bold{I}_{n+1},\bold{I}_{n+1},\bold{I}_{n}) ,0,C_S(\bold{I}_{n+2},\bold{I}_{n+2},\bold{I}_{n}),\nonumber \\
														 & & C_S(\bold{I}_{n+2},\bold{I}_{n+2},\bold{I}_{n+1})\Big).\label{eqth1}
\end{eqnarray}
On the other hand, using the triangle inequality, we obtain
\begin{eqnarray}
C_S(\bold{I}_{n},\bold{I}_{n},\bold{I}_{n+2}) & \leq & \alpha\big(C_S(\bold{I}_{n},\bold{I}_{n},\bold{I}_{n+1})\big)+\alpha\big(C_S(\bold{I}_{n},\bold{I}_{n},\bold{I}_{n+1})\big)+\alpha\big(C_S(\bold{I}_{n+2},\bold{I}_{n+2},\bold{I}_{n+1})\big)\nonumber \\
                         & = & 2\alpha\big(C_S(\bold{I}_{n},\bold{I}_{n},\bold{I}_{n+1})\big)+\alpha\big(C_S(\bold{I}_{n+2},\bold{I}_{n+2},\bold{I}_{n+1})\big).\label{eqth2}
\end{eqnarray}
By applying the condition $iii)$ to the equation (\ref{eqth2}) we get
\begin{equation}
C_S(\bold{I}_{n},\bold{I}_{n},\bold{I}_{n+2})  \leq  2C_S(\bold{I}_{n},\bold{I}_{n},\bold{I}_{n+1})+C_S(\bold{I}_{n+2},\bold{I}_{n+2},\bold{I}_{n+1}).\label{eqth3}
\end{equation}
Then, using (\ref{eqth1}) and (\ref{eqth3}) we can affirm that the function $M_f$ satisfies the condition $(M_1)$ that is 
\begin{eqnarray}
C_S(\bold{I}_{n+1},\bold{I}_{n+1},\bold{I}_{n+2})  & \leq &  rC_S(\bold{I}_{n},\bold{I}_{n},\bold{I}_{n+1})\nonumber \\
                              & \leq &  r^n C_S(\bold{I}_{0},\bold{I}_{0},\bold{I}_{1}), \  \  \mbox{ for } r \in (0,1).\label{eqth4}
\end{eqnarray}
Now, let  $n,m \in \mathbb{N}$, $n<m$ by using  the triangle inequality similarly to the Theorem \ref{thm1} we obtain
\begin{eqnarray*}
C_S(\bold{I}_{n},\bold{I}_{n},\bold{I}_{m}) & \leq & \alpha\big(2C_S(\bold{I}_{n},\bold{I}_{n},\bold{I}_{n+1})\big)+\alpha\big(C_S(\bold{I}_{m},\bold{I}_{m},\bold{I}_{n+1})\big)\\
                       & \leq & \alpha\big(2C_S(\bold{I}_{n},\bold{I}_{n},\bold{I}_{n+1})\big)+\alpha\Big(2\alpha(C_S(\bold{I}_{m},\bold{I}_{m},\bold{I}_{n+2}))+\alpha(C_S(\bold{I}_{n+1},\bold{I}_{n+1},\bold{I}_{n+2})) \Big)\\
											 & \leq & 2\alpha\big(C_S(\bold{I}_{n},\bold{I}_{n},\bold{I}_{n+1})\big)+\alpha^2(C_S(\bold{I}_{n+1},\bold{I}_{n+1},\bold{I}_{n+2}))+2\alpha^2\Big((C_S(\bold{I}_{m},\bold{I}_{m},\bold{I}_{n+2}) \Big)\\
											 & \leq & 2\alpha\big(C_S(\bold{I}_{n},\bold{I}_{n},\bold{I}_{n+1})\big)+\alpha^2(C_S(\bold{I}_{n+1},\bold{I}_{n+1},\bold{I}_{n+2}))+2\alpha^2\Big(2\alpha(C_S(\bold{I}_{m},\bold{I}_{m},\bold{I}_{n+3}))\\
					 						& + & \alpha (C_S(\bold{I}_{n+2},\bold{I}_{n+2},\bold{I}_{n+3}))\Big)\\
											 & \leq & 2\alpha\big(C_S(\bold{I}_{n},\bold{I}_{n},\bold{I}_{n+1})\big)+\alpha^2(C_S(\bold{I}_{n+1},\bold{I}_{n+1},\bold{I}_{n+2}))+ 2\alpha^3 (C_S(\bold{I}_{n+2},\bold{I}_{n+2},\bold{I}_{n+3}))\\
											& + & 2^2\alpha^3\Big(2\alpha(C_S(\bold{I}_{m},\bold{I}_{m},\bold{I}_{n+3}))\\
											& \vdots &\\
											& \leq & 2\alpha\big(C_S(\bold{I}_{n},\bold{I}_{n},\bold{I}_{n+1})\big)+\alpha^2(C_S(\bold{I}_{n+1},\bold{I}_{n+1},\bold{I}_{n+2}))+ 2\alpha^3 (C_S(\bold{I}_{n+2},\bold{I}_{n+2},\bold{I}_{n+3}))\\
											& + & \displaystyle\sum_{k=n+3}^{m-2}2^{k-n-2}\alpha^{k-n+1}\Big(C_S(\bold{I}_k,\bold{I}_k,\bold{I}_{k+1})\Big)+2^{m-n-3}\alpha^{m-n-1}\Big(C_S(\bold{I}_m,\bold{I}_m,\bold{I}_{m-1})\Big)													
\end{eqnarray*}
Using (\ref{eqth4}) we get
\begin{eqnarray*}
C_S(\bold{I}_{n},\bold{I}_{n},\bold{I}_{m}) & \leq &  2\alpha\big(r^nC_S(\bold{I}_{0},\bold{I}_{0},\bold{I}_{1})\big)+\alpha^2(r^{n+1}C_S(\bold{I}_{0},\bold{I}_{0},\bold{I}_{1}))+ 2\alpha^3 (r^{n+2}C_S(\bold{I}_{0},\bold{I}_{0},\bold{I}_{1}))\\
											& + & \displaystyle\sum_{k=n+3}^{m-2}2^{k-n-2}\alpha^{k-n+1}\Big(r^kC_S(\bold{I}_0,\bold{I}_k,0_{1})\Big)+2^{m-n-3}\alpha^{m-n-1}\Big(r^{m-1}C_S(\bold{I}_0,\bold{I}_0,\bold{I}_{1})\Big).		
\end{eqnarray*}
By conditions $ii)$ and $iv)$ we can observe that $\displaystyle\lim_{n,m\rightarrow \infty} C_S(\bold{I}_{n},\bold{I}_{n},\bold{I}_{m}) =0$. Therefore $\{ \bold{I}_n\}$ is a Cauchy sequence. By virtue of completeness, there is a $\bold{I} \in \mathcal{A}$ such that $\bold{I}_n \rightarrow \bold{I}$ when $n$ goes to $\infty$.\\
Let us prove that $\bold{I}$ is FP of $F$.
 \begin{eqnarray*}
C_S(\bold{I}_{n+1},\bold{I}_{n+1},F\bold{I}) & = & C_S(F\bold{I}_{n},F\bold{I}_{n},F\bold{I})\\
                        & \leq & M_f\Big(C_S(\bold{I}_{n},\bold{I}_{n},\bold{I}), C_S(F\bold{I}_{n},F\bold{I}_{n},\bold{I}_n),C_S(F\bold{I}_{n},F\bold{I}_{n},\bold{I}),\\
												&  & C_S(F\bold{I},F\bold{I},\bold{I}_n),C_S(F\bold{I},F\bold{I},\bold{I})  \Big)\\
												& = & M_f\Big(C_S(\bold{I}_{n},\bold{I}_{n},\bold{I}), C_S(\bold{I}_{n+1},\bold{I}_{n+1},\bold{I}_n),C_S(\bold{I}_{n+1},\bold{I}_{n+1},\bold{I}),\\
												&  & C_S(F\bold{I},F\bold{I},\bold{I}_n),C_S(F\bold{I},F\bold{I},\bold{I})  \Big).											
\end{eqnarray*}
As $n \longrightarrow \infty$,  taking the limit gives us
\begin{eqnarray*}
C_S(\bold{I},\bold{I},F\bold{I}) & \leq  & M_f\Big(0,0,0,C_S(F\bold{I},F\bold{I},\bold{I}) , C_S(F\bold{I},F\bold{I},\bold{I})  \Big)\\
            & =  & M_f\Big(0,0,0,C_S(\bold{I},\bold{I},F\bold{I}) , C_S(\bold{I},\bold{I},F\bold{I})  \Big).
\end{eqnarray*}
Given that  $M_f$ satisfies the condition $(M_1)$, we have  $C_S(\bold{I},\bold{I},F\bold{I})\leq r.0$ which gives us $C_S(\bold{I},\bold{I},F\bold{I})=0$. Therefore $\bold{I}=F\bold{I}$ that is $\bold{I}$ is a FP of $F$.\\
Let $\bold{I}^1, \bold{I}^2$ two FP of $F$.
\begin{eqnarray}
C_S(\bold{I}^1,\bold{I}^1, \bold{I}^2) & = & C_S(F\bold{I}^1,F\bold{I}^1, F\bold{I}^2)\\
                  & \leq & M_f\Big(C_S(\bold{I}^1,\bold{I}^1, \bold{I}^2), C_S(F\bold{I}^1,F\bold{I}^1, \bold{I}^1) , C_S(F\bold{I}^1,F\bold{I}^1, \bold{I}^2) , \nonumber \\
									&   &    C_S(F\bold{I}^2,F\bold{I}^2, \bold{I}^1) ,C_S(F\bold{I}^2,F\bold{I}^2, \bold{I}^2)  \Big) \nonumber \\
									& = & M_f\Big(C_S(\bold{I}^1,\bold{I}^1, \bold{I}^2), C_S(\bold{I}^1,\bold{I}^1, \bold{I}^1) , C_S(\bold{I}^1,\bold{I}^1, \bold{I}^2) ,  \nonumber \\
									&   &    C_S(\bold{I}^2,\bold{I}^2, \bold{I}^1) ,C_S(\bold{I}^2,\bold{I}^2, \bold{I}^2)  \Big) \nonumber \\
									& = & M_f\Big(C_S(\bold{I}^1,\bold{I}^1, \bold{I}^2), 0 , C_S(\bold{I}^1,\bold{I}^1, \bold{I}^2) ,C_S(\bold{I}^2,\bold{I}^2, \bold{I}^1),0\Big) \nonumber \\
									& = & M_f\Big(C_S(\bold{I}^1,\bold{I}^1, \bold{I}^2), 0 , C_S(\bold{I}^1,\bold{I}^1, \bold{I}^2) ,C_S(\bold{I}^1,\bold{I}^1, \bold{I}^2),0\Big).\label{eqth6}
\end{eqnarray}
Since $M_f$ satisfies the condition $(M_2)$ and giving (\ref{eqth6}) we obtain that $C_S(\bold{I}^1,\bold{I}^1, \bold{I}^2)=0$ then $\bold{I}^1= \bold{I}^2$.
\end{proof}

\begin{corollary}
By choosing $M(t_1,t_2,t_3,t_4,t_5)=rt_1$ in Theorem \ref{thm2} we obtain Theorem \ref{thm1}.
\end{corollary}

The following corollary is a FP result for Kannan contraction in a  $CSM$.
\begin{corollary}
Let $(\mathcal{A},C_S)$ be a  symmetric complete $CSM$ with non-constant control function  $\alpha:[0,\infty)\longrightarrow [0,\infty)$. Let $F: \mathcal{A} \rightarrow \mathcal{A}$ be a  mapping  such that for all $o,h\in \mathcal{A}$ and some $a\in[0,\dfrac{1}{2})$
\begin{equation}
C_S(Fo,Fo,Fh)\leq a\big[C_S(Fo,Fo,o)+  C_S(Fh,Fh,h)  \big]. \label{kannan}
\end{equation}
Hence, $F$ has a unique FP.
\end{corollary}

\begin{proof}
The key idea is  to apply the result of Theorem \ref{thm2} by selecting the appropriate function $M_f$. We choose $M_f(o,h,w,s,t)=a.(h+t)$  for some $0<a\leq \dfrac{1}{2}$ and all $o,h,w,s,t\in \mathbb{R}_+$. Indeed, by opting  for $M_f(o,h,w,s,t)=a.(h+t)$ in the contraction (\ref{M-contr})  we obtain 
\begin{equation*}
C_S(Fo,Fo,Fh)\leq a\big[C_S(Fo,Fo,o)+  C_S(Fh,Fh,h)  \big].
\end{equation*}
It is easy to see that $M_f$ is continuous. Let us verify that $M_f$ satisfies the conditions $(M_1)$ and $(M_2)$.\\
For all $o,h,w\in \mathbb{R}_+$ we have $M_f(o,o,0,w,h)=a.(o+h)$. So if $h\leq M_f(o,o,0,w,h)$ with $w\leq 2o+h$, then $h\leq a(o+h)$ which imply that $h\leq \dfrac{a}{1-a}o$ where $\dfrac{a}{1-a}<1$. \\ 
Therefore, $M_f$ satisfies the condition $(M_1)$.\\
Next, if $h\leq M_f(h,0,h,h,0)=a.(0+0)=0$ then $h=0$ and $(M_2)$ holds.\\
Finally, the chosen function $M_f(o,h,w,s,t)=a.(h+t) \in \mathcal{M}_f $ and by applying Theorem \ref{thm2} it follows  that $F$ possesses  a unique FP.
\end{proof}

\begin{corollary}
Let $(\mathcal{A},C_S)$ be a  symmetric complete $CSM$ with non-constant control function  $\alpha:[0,\infty)\longrightarrow [0,\infty)$. Let $F: \mathcal{A} \rightarrow \mathcal{A}$ be a  mapping  such that $\forall$ $o,h\in \mathcal{A}$ and some $a\in [0,1)$
\begin{equation}
C_S(Fo,Fo,Fh)\leq a\max\big[C_S(Fo,Fo,o)+  C_S(Fh,Fh,h)  \big]. \label{Bianchini}
\end{equation}
Hence, $F$ has a unique FP.
\end{corollary}

\begin{proof}
The claim is a direct consequence  of Theorem \ref{thm2} with $M_f(o,h,w,s,t)=a\max\{h,t\}$ for some $a\in [0,1)$ and all $o,h,w,s,t\in \mathbb{R}_+$.  Indeed, $M_f$ is continuous. \\
Let us check the condition $(M_1)$. We have $M_f(o,o,0,w,h)=a\max\{o,h\}$. So, if $h\leq M_f(o,o,0,w,h)$ with $w\leq o+2h$, then $h\leq a.e$ or $h\leq a.h$. \\
Thus,  $h\leq a.o$ and $F$ fulfills  the condition $(M_1)$.\\

Next, if $h\leq M_f(h,0,h,h,0)=a\max \{ 0,0\}=0$, then $h=0$ and the condition $(M_2)$ holds.\\
Finally, $M_f(o,h,w,s,t)=a\max\{h,t\}\in \mathcal{M}_f$ and by Theorem \ref{thm2}, $F$ has a unique FP.
\end{proof}

\section{Application}
To showcase the broad applicability of the FP result in different mathematical contexts, we use it to solve nth-degree polynomial equations through Theorem \ref{thm1}. By thoroughly verifying the conditions outlined in Theorem \ref{thm1}, we ensure that the FP serves as both a valid and effective approach for solving nth-degree polynomial equations, as demonstrated in Theorem \ref{app_thm}.
\begin{theorem}\label{app_thm}
 Let $m\in \mathbb{N}$, for $m\geq 3$, the  following equation
\begin{equation}
v^m-(m^4-1)v^{m+1}-m^4v+1=0 \label{app}
\end{equation}
has a unique solution in the interval $[0,1]$.
\end{theorem}
\begin{proof}
Let $\mathcal{A}=[0,1]$ and for any $p,s,q\in \mathcal{A}$ define the the function $C_S:\mathcal{A}^3\longrightarrow [0,\infty)$ by
\begin{equation*}
C_S(p,s,q)=|p-s|+|s-q|.
\end{equation*}
 The composed function $\alpha: [0,\infty) \longrightarrow [0,\infty)$ is given by $\alpha(p)=2\sqrt{p}$.\\ 
First of all, let shown that $(\mathcal{A}, C_S)$ is a complete $CSM$. We will show that the triangle inequality holds.\\
Let $p,s,q,c\in [0,1]$, we have 
\begin{eqnarray}
C_S(p,s,q)=|p-s|+|s-q| & = & |p-c+c-s|+|s-c+c-q| \nonumber \\
                       & \leq & |p-c|+2|s-c|+|q-c|. \nonumber 
\end{eqnarray}
Knowing that $p\leq \sqrt{p}$ $\forall$ $0\leq p \leq 1$ and $C_S(p,p,c)=2|p-c|$, we get 
\begin{eqnarray}
C_S(p,s,q) & \leq & \sqrt{|p-c|}+2\sqrt{|s-c|}+\sqrt{|q-c|} \nonumber \\
           & \leq & 2\sqrt{2|p-c|}+2\sqrt{2|s-c|}+2\sqrt{2|q-c|} \nonumber \\
					 & = & \alpha(2|p-c|)+\alpha(2|s-c|)+\alpha(2|q-c|) \nonumber \\
					& = & \alpha(C_S(p,p,c))+\alpha(C_S(s,s,c))+\alpha(C_S(q,q,c)). \nonumber
\end{eqnarray}
Then, $([0,1], C_S)$ is a complete $CSM$.\\
Define $F:\mathcal{A} \longrightarrow \mathcal{A}$ such that for all $p\in \mathcal{A}$,
\begin{equation}
Fp=\dfrac{p^m+1}{(m^4-1)p^m+m^4}. \label{app1}
\end{equation}
Given that $m\geq 3$, we will set $m= 3$ to simplify the computation. However, using this approach, it can be demonstrated that the results hold for any $m\geq 3$. Therefore, equation (\ref{app1}) becomes
\begin{equation}
Fp=\dfrac{p^3+1}{80p^3+81}  . \label{app2}
\end{equation}
It is easy to see that if $F$ has a FP that is $Fp=p$, then $\dfrac{p^m+1}{(m^4-1)p^m+m^4}=p$ imply that $p$ is a solution of (\ref{app}).

We will now  show that $F$ meets the contraction (\ref{banach-contr}) used in Theorem \ref{thm1}. Indeed, 
\begin{eqnarray}
C_S(Fp,Fs,Fq) & = & |Fp-Fs|+|Fs-Fq|\nonumber \\
              & = & \Big|\dfrac{p^3+1}{80p^3+81}- \dfrac{s^3+1}{80s^3+81} \Big| +\Big| \dfrac{s^3+1}{80s^3+81}-\dfrac{q^3+1}{80q^3+81}\Big| \nonumber \\
							& = & \Big|\dfrac{p^3-s^3}{(80p^3+81)(80s^3+81)} \Big|+ \Big|\dfrac{s^3-q^3}{(80s^3+81)(80q^3+81)} \Big| \nonumber \\
							& \leq & \dfrac{|p-s|}{81}	+ \dfrac{|s-q|}{81}			\nonumber\\
							& \leq & \dfrac{1}{81} \Big(|p-s|	+ |s-q| \Big)		\nonumber\\
C_S(Fp,Fs,Fq)			& = & \dfrac{1}{81} \Big(C_S(p,s,q) \Big).			\label{app3}
\end{eqnarray}
Thus, $F$ satisfies (\ref{banach-contr}) with $r=\dfrac{1}{81}$.\\
Alternatively, the composed function $\alpha(p)=2\sqrt{p}$ satisfies the condition
\begin{equation}
\alpha(kp+s)\leq \alpha(p)+k\alpha(s) \  \ \forall p\in [0,1].\label{app4}
\end{equation}
Now, we are left with the third condition of the Theorem \ref{thm1}.
\begin{eqnarray}
\displaystyle\lim_{n\rightarrow \infty}\displaystyle\sum_{k=n+3}^{m-2}2^{k-n-1}\alpha^{k-n+1}\Big(r^kC_S(p_0,p_0,p_{1})\Big)+2^{m-n-2}\alpha^{m-n-1}\Big(r^mC_S(p_0,p_0,p_{1})\Big) & = &\nonumber\\
\displaystyle\lim_{n\rightarrow \infty}\displaystyle\sum_{k=n+3}^{m-2}2^{k-n-1}\Big(\sqrt{r^kC_S(p_0,p_0,p_{1})}\Big)^{k-n+1}+2^{m-n-2}\Big(\sqrt{r^mC_S(p_0,p_0,p_{1})}\Big)^{m-n-1} & = &\nonumber\\
\displaystyle\lim_{n\rightarrow \infty}\displaystyle\sum_{k=n+3}^{m-2}2^{k-n-1}\Big(\sqrt{r^k2|p_0-p_{1}|)}\Big)^{k-n+1}+2^{m-n-2}\Big(\sqrt{r^m2|p_0-p_{1}|)}\Big)^{m-n-1}=0. \label{app5}
\end{eqnarray}
Finally, thanks to  the results  (\ref{app3}),  (\ref{app4}) and  (\ref{app5}), all the hypotheses of Theorem \ref{thm1} are satisfied. Therefore, $F$ has a unique FP in $\mathcal{A}$, which is a unique real solution of equation (\ref{app}).
\end{proof}

\section{Conclusion and Perspectives}
In this paper, we have introduced the concept of composed $S$-metric spaces, which extends the classical framework of $S$-$MS$ by incorporating a composed function.  We have established  the existence and uniqueness of FP under appropriate conditions.  Our results encompass and generalize some existing results in the literature, offering new insights and tools for research in metric and topological spaces. Through illustrative examples, we have shown the practical relevance and applicability of composed $S$-metric spaces in solving $n$-th degree polynomial equations. \\

The concept of composed $S$-metric spaces opens up several avenues for future research. First, it would be valuable to explore additional FP theorems within this framework. Another promising direction is the application of the proposed approach to the Riemann–Liouville fractional integral equations and to other differential equations. Furthermore, future work could investigate the relationship between composed $S$-metric spaces and other generalized metric spaces, such as fuzzy or probabilistic metric spaces, to develop a broader unified theory of metrics and their applications.






\end{document}